\documentclass{amsart}
\usepackage{url,hyperref}
\usepackage[noadjust]{cite}

\theoremstyle{plain}

\newtheorem{theorem}{Theorem}[section]
\newtheorem{proposition}[theorem]{Proposition} 
\newtheorem{corollary}[theorem]{Corollary}
\newtheorem{lemma}[theorem]{Lemma}
\newtheorem{problem}[theorem]{Problem}

\theoremstyle{definition}
\newtheorem{definition}[theorem]{Definition}
\newtheorem{example}[theorem]{Example}
\newtheorem{remark}[theorem]{Remark}

\theoremstyle{remark}
\newtheorem{claim}[theorem]{Claim}

\newcommand\Aut{\operatorname{Aut}}
\newcommand\id{\operatorname{id}}
\newcommand\Bir{\operatorname{Bir}}
\newcommand\Gal{\operatorname{Gal}}
\newcommand\ch{\operatorname{char}}

\begin{document}

\title[Simultaneous Galois points for a reducible plane curve]{Simultaneous Galois points for a reducible plane curve consisting of nonsingular components}

\author[A. Ikeda]{Aki Ikeda}
\address[A. Ikeda]{Department of Mathematics, Graduate School of Science and Technology, Niigata University, Niigata 950-2181, Japan.} 
\email{ikeda@m.sc.niigata-u.ac.jp}

\author[T. Takahashi]{Takeshi Takahashi}
\address[T. Takahashi]{Education Center for Engineering and Technology, Faculty of Engineering, Niigata University, Niigata 950-2181, Japan}
\email{takeshi@eng.niigata-u.ac.jp}


\thanks{We thank Prof.~Satoru Fukasawa and Prof.~Kei Miura for their helpful advice. We would like to thank Editage (\url{www.editage.com}) for English language editing. This work was supported by JSPS KAKENHI Grant Number JP19K03441 and JST SPRING Grant Number JPMJSP2121.}


\begin{abstract}
Yoshihara's definition of Galois points for irreducible plane curves is extended to reducible plane curves. We also define simultaneous Galois points, weakening the conditions of the definition. We studied the number of simultaneous Galois points for a reduced plane curve with nonsingular components.
\end{abstract}

\maketitle


\section{Introduction} \label{Section Introduction}

The Galois points for an irreducible plane curve are as follows: 
\begin{definition}[\cite{zbMATH01441808}, \cite{zbMATH01643659}]
	Let $C \subset \mathbb{P}^2$ be an irreducible plane curve of degree $d$ over algebraically closed field $k$. For a point $P \in \mathbb{P}^2$, let $\pi_P:C \dashrightarrow \mathbb{P}^1$ be the projection with the center $P$. The point $P$ is called a {\itshape Galois point} if the field extension $k(C)/\pi_P^*k(\mathbb{P})$ given by the projection $\pi_P$ is Galois. We denote by $\Delta_{\mathrm{in}}(C)$ (resp. $\Delta_{\mathrm{out}}(C)$) the set of all Galois points on $C$ (resp. not on $C$), and $\Delta(C)=\Delta_{\mathrm{in}}(C) \cup \Delta_{\mathrm{out}}(C)$.
\end{definition}

Numerous studies have been conducted on the Galois points. For instance, Galois points have been enumerated for non-singular plane curves (\cite{zbMATH01441808}, \cite{zbMATH01643659}, \cite{zbMATH05126212}, \cite{zbMATH05225736, zbMATH05202381, zbMATH05241281, https://doi.org/10.48550/arxiv.1011.3648, zbMATH05721541, zbMATH06199262, zbMATH06244917}). In \cite{https://doi.org/10.48550/arxiv.2210.02076}, Fukasawa demonstrated that, for a (possibly singular) curve, the number of outer Galois points, which refer to Galois points not on $C$, is at most three when $\mathrm{char}(k) = 0$ and $d \geq 3$. Yoshihara defined Galois points for an irreducible hypersurface in high-dimensional projective space (\cite{zbMATH01675821}, \cite{zbMATH01956748}). Furthermore, Yoshihara defined a Galois subspace for a projective variety in $\mathbb{P}^N$ and a Galois embedding for a nonsingular projective variety and investigated them (\cite{zbMATH05045279}, \cite{zbMATH05214476}, etc.). Fukasawa, Miura, and the second author have introduced quasi-Galois points, which are a generalization of Galois points for irreducible plane curves, and studied their number, distribution, and Galois groups (\cite{zbMATH07199975}, \cite{https://doi.org/10.48550/arxiv.2211.16033}). For further research trends related to Galois points, refer to the ``List of Problems'' authored by Yoshihara and Fukasawa (\cite{ListOfProblems}).

In this study, we defined Galois points and investigated the number of Galois points for reducible plane algebraic curves. Hereafter, let $k$ be an algebraically closed field with characteristic $0$. Let $C$ be a reduced projective curve in $\mathbb{P}^2$ over $k$. Let $C=C_1 \cup \dots \cup C_n$ be the irreducible decomposition of $C$ and $K_i$ be the function field of $C_i$. Let $R(C)$ be the function ring of $C$ (for the definition, see Section 7.1.3 in Chapitre Premier of \cite{MR217083} or Exercise~I-21 in \cite{zbMATH01260607}). Here, $R(C)$ is isomorphic with the direct product $K_1 \times \dots \times K_n$. For reducible plane curves, the function rings are considered alternatives to the function fields. We consider a projection $\pi_P: C \dashrightarrow \mathbb{P}^1$ with the center $P \in \mathbb{P}^2$. Using $\pi_P$, we obtain the field extensions $\pi_P^*: k(\mathbb{P}^1) \hookrightarrow K_i$ and the extension of $k$-algebras $\pi_P^*: k(\mathbb{P}^1) \hookrightarrow R(C)$. Let $K_P$ denote the image of $\pi_P^*: k(\mathbb{P}^1) \hookrightarrow R(C)$. 

The Galois extension for algebra over a field is as follows:

\begin{definition}[Chapter V in \cite{zbMATH01201676}]
Let $F$ be a field. A finite-dimensional $F$-algebra $L$ is called an {\itshape \'{e}tale} $F$-algebra if $L \simeq K_1 \times \cdots \times K_r$ (as $F$-algebra) for certain finite separable field extensions $K_1, \dots, K_r$ of $F$. The {\itshape \'{e}tale} $F$-algebra $L$ endowed with an action by a finite group $G$ of $F$-algebra automorphisms, is called a {\itshape $G$-algebra} over $F$. The $G$-algebra over $F$ is said to be {\itshape Galois} if $|G|=\dim_F L$ and $L^G =F$ hold, where $L^G = \{ x \in L \mid g(x)=x \text{ for all } g \in G \}$. 
\end{definition}

Following Yoshihara's definition of the Galois points, we define them for reduced plane curves as follows:

\begin{definition} \label{Def of G-Galois point}
	A point $P \in \mathbb{P}^2$ is called a {\itshape Galois point with a Galois group $G$} if the function ring $R(C)$ is a Galois $G$-algebra over $K_P$. 
\end{definition}

Based on the fact that every Galois $G$-algebra extension $L \simeq K_1 \times \cdots \times K_r$ satisfies that each $K_i/F$ is Galois and $K_i \simeq K_j$ ($\forall i, j$) as $F$-algebra (Proposition 18.18 in \cite{zbMATH01201676}), we provide a slightly weakened concept of the Galois point definition. 

\begin{definition} \label{Def of SG point}
	A point $P \in \mathbb{P}^2$ is said to be {\itshape simultaneous Galois} (``SG'' for short) if each field extension $k(C_i)/\pi^*_P k(\mathbb{P}^1)$ ($i=1, \dots, n$) is Galois and $k(C_i) \simeq k(C_j)$  ($\forall i,j$) as $k(\mathbb{P}^1)$-algebra. The SG point is considered {\itshape inner} (resp. {\itshape outer}) if $P \in \cap_{i=i}^{n} C_i$ (resp. $P \notin \cup_{i=i}^{n} C_i$). 
	We denote by $S\Delta(C_1, \dots, C_n)$ (resp. $S\Delta_{\mathrm{in}}(C_1, \dots, C_n)$, $S\Delta_{\mathrm{out}}(C_1, \dots, C_n)$) the set of all SG points (resp. inner SG points, outer SG points). 
\end{definition}

Although not discussed here, extensions of the definition ``SG point'' (for example, SG points for high-dimensional hypersurfaces,  SG subspaces, quasi-SG points, etc.) can also be considered. In this study, we examine the number of SG points in the simplest case. Our main results, as discussed in later sections, are as follows: 

\begin{theorem}
	Let $C = C_1 \cup C_2$ be a reduced curve in $\mathbb{P}^2$. 
	\begin{enumerate}
		\item (Theorem~\ref{Theorem SG points two quadrics}) Assume that $C_1$ and $C_2$ are nonsingular quadric curves. Then the number of outer SG points equals $0$, $1$, $3$, or $6$. 
		\item (Theorem~\ref{Theorem the number of outer SG points d>2}) Assume that $C_1$ and $C_2$ are nonsingular curves of same degree $d$, where $d \geq 3$. Then the number of outer SG points is at most one.  
		\item (Lemma~\ref{Lemma number of SG points}) Assume that $C_1$ and $C_2$ are nonsingular curves of same degree $d$, where $d \geq 5$. Then the number of inner SG points is at most one. 
		\item (Theorem~\ref{Theorem the number of inner SG points}) Assume that $C_1$ and $C_2$ are nonsingular quartic curves. Then the number of inner SG points is $0$, $1$, or $2$. 
		\item (Theorem~\ref{Theorem C has inner SG and outer SG})
			Assume that $C_1$ and $C_2$ are nonsingular curves of the same degree $d$, where $d \geq 4$. Then $S\Delta_{\mathrm{in}}(C_1,C_2)=\emptyset$ or $S\Delta_{\mathrm{out}}(C_1,C_2)=\emptyset$.  
	\end{enumerate}
\end{theorem}

We begin with the preliminary results in Section 2. In Section 3, we analyze the outer SG points of a plane curve comprising two quadrics. In Section 4, we explore the outer SG points of a plane curve composed of two curves with degrees greater than two. Finally, in Section 5, we examine the inner SG points of a plane curve consisting of two quartic curves. In each section, we present the relevant theorem, proof, and examples.


\section{Preliminaries}

To understand the Galois points and SG points, we first present the preliminary results. 

We summarize the results for the Galois points of nonsingular plane curves. 
	
\begin{theorem}[\cite{zbMATH01441808}, \cite{zbMATH01643659}, \cite{https://doi.org/10.48550/arxiv.1011.3648}, \cite{zbMATH05721541}] \label{Theorem Galois points}
	Let $C \subset \mathbb{P}^2$ be a nonsingular projective curve of degree $d \geq 3$ over an algebraically closed field $k$ with characteristic $0$. 
	Then:
		\begin{enumerate}
		\item for a Galois point $P$, the Galois group of the extension $k(C)/\pi^*_P k(\mathbb{P}^1)$ is cyclic;
		\item $\# \Delta_{\mathrm{out}}(C) = 0,1$ or $3$; 
		\item $\# \Delta_{\mathrm{out}}(C) = 3$ if and only if $C$ is projectively equivalent to the Fermat curve $F_d: X^d + Y^d + Z^d = 0$, where $(X:Y:Z)$ is a system of homogeneous coordinates on $\mathbb{P}^2$; 
		\item for the Fermat curve $F_d: X^d + Y^d + Z^d = 0$, $\Delta_{\mathrm{out}}(C)=\{(0:0:1),(0:1:0),(1:0:0)\}$;
		\item if $d \geq 5$, then $\# \Delta_{\mathrm{in}}(C) = 0$ or $1$; 
		\item if $d =4$, then $\# \Delta_{\mathrm{in}}(C) = 0,1$ or $4$;
		\item $\# \Delta_{\mathrm{in}}(C) = 4$ if and only if $C$ is projectively equivalent to the curve $XY^3 + X^4 + Z^4 = 0$; 
		\item if $d \geq 4$, $\# \Delta_{\mathrm{in}}(C) \geq 1$ and $\# \Delta_{\mathrm{out}}(C) \geq 1$, then $C$ is projectively equivalent to the curve $XY^{d-1} + X^d + Z^d = 0$;
		\item for the curve $XY^3 + X^4 + Z^4 = 0$, $\Delta_{\mathrm{in}}(C) = \{(0:1:0),(-1:1:0),(-\omega:1:0), (-\omega^2:1:0)\}$ and $\Delta_{\mathrm{out}}(C)=\{(0:0:1)\}$, where $\omega$ is a primitive cubic root of unity; 
		\item for the curve $XY^{d-1} + X^d + Z^d = 0$ ($d \geq 5$), $\Delta_{\mathrm{in}}(C) =\{(0:1:0)\}$ and $\Delta_{\mathrm{out}}(C)=\{(0:0:1)\}$.
	\end{enumerate} 
\end{theorem}

Here, we provide some elementary facts about the Galois and SG points. Let $C=\cup_{i=1}^n C_i \subset \mathbb{P}^2$ be a reduced curve and $P \in \mathbb{P}^2$ be a point. The projection $\pi_P: C \dashrightarrow \mathbb{P}^1$ with the center $P$ induces a monomorphism $\pi^*_P: k(\mathbb{P}^1) \hookrightarrow R(C)$, which is regarded as a $k(\mathbb{P}^1)$-algebra $R(C)$, where $R(C)$ is the function ring of $C$. Note that $R(C)= k(C_1) \times \dots \times k(C_n)$. We denote by $K_P$ the image of $\pi^*_P$. We denote by $\Bir(C)$ the group of birational maps from $C$ to $C$, and by $\Bir(\pi_P)$ the subgroup $\{ f \in \Bir(C) \mid \pi_P \circ f = \pi_P \}$. Let $G$ be a finite subgroup of $\Aut_{K_P}(R(C))$, where $\Aut_{K_P}(R(C))$ is the automorphism group of $K_P$-algebra $R(C)$. Group $G \subset \Aut_{K_P}(R(C))$ corresponds to a group $G^{\mathrm{op}} \subset \Bir(\pi_P)$. Alternatively, an element $g \in G$, which is an automorphism $g:R(C) \rightarrow R(C)$ over $K_P$, yields a birational map $f:C \dasharrow C$. This $g \mapsto f$ yields anti-isomorphism $G \rightarrow G^{\mathrm{op}}$. An element $f \in G^{\mathrm{op}}$ is expressed as a tuple $(\sigma_f; \varphi_{f,1}, \dots, \varphi_{f,n})$ where 
\begin{enumerate}
	\item $\sigma_f$ is a bijection $\{1,\dots, n \} \rightarrow \{1,\dots, n\}$; 
	\item $\varphi_{f, i}$ ($i=1, \dots, n$) is a birational map $C_i \dasharrow C_{\sigma_f(i)}$.  
\end{enumerate}
For $f,g \in G^{\mathrm{op}}$, we have that $\sigma_{f g} = \sigma_{f} \circ \sigma_{g}$ and $\varphi_{f g, i} = \varphi_{f, \sigma_{g}(i)} \circ \varphi_{g,i}$. The following lemma is given by the definition of the Galois point: 
\begin{lemma}
	A point $P$ is Galois with Galois group $G$ if and only if
\begin{enumerate}
	\item $\# G^{\mathrm{op}}= \sum_{i=1}^n \deg \pi_{P, i}$, where $\pi_{P, i}$ is the projection $\pi_P:C_i \dasharrow \mathbb{P}^1$; and
	\item $\{ z \in R(C) \mid f^*(z) = z \text{ for all } f \in G^{\mathrm{op}} \} = K_P$, where $f^*(z)$ means the pullback of $z$ by $f$. 
\end{enumerate} 
\end{lemma}

We have the following algebraic facts about the Galois $G$-extension:
\begin{proposition}[Proposition~18.18 in \cite{zbMATH01201676}] \label{Proposition Galois G algebra}
Let $L \simeq K_1 \times \cdots \times K_r$ be a Galois $G$-algebra over $F$, where $K_i$ is a separable field extension of $F$. Then $K_i \simeq K_j$ as $F$-algebra for all $i,j$, and the field extensions $K_i/F$ are Galois.
\end{proposition}

The following lemma follows immediately from the above proposition.

\begin{lemma} \label{Lemma G-Galois is SG}
	Let $C \subset \mathbb{P}^2$ be a reduced plane curve. We assume that the point $P \in \mathbb{P}^2$ is a Galois point for $C$. Then $P$ is an SG point. 
\end{lemma}

Let us consider the converse of Lemma~\ref{Lemma G-Galois is SG}. 
\begin{remark}[cf. Example 18.16 in \cite{zbMATH01201676}] \label{Remark direct product of cyclic extensions}
	Let $F$ and $K$ be fields. We assume that $K$ is a Galois extension over $F$ of degree $d$. Let $L$ be the algebra $K \times \dots \times K$ (the direct product of $n$ fields). Then $L$ is regarded as $F$-algebra with $t \cdot (k_1, \dots, k_n) = (tk_1, \dots, tk_n )$ where $t \in F$ and $(k_1, \dots, k_n) \in L$, and $\dim_F L = dn$. 	
	\begin{enumerate}
		\item Assume that $\Gal (K/F)$ is cyclic. Let $\alpha$ be a generator of $\Gal (K/F)$. Let $\sigma$ be the automorphism of $L$ over $F$ defined by
			\[ \sigma ((k_1, \dots, k_n)) = (\alpha(k_n), k_1, \dots, k_{n-1}), \]
			and $G$ be the group generated by $\sigma$. Then $G$ is a cyclic group of order $dn$ and $L$ is a Galois $G$-algebra over $F$. 
		\item  Let $\sigma_{\alpha}$ ($\alpha \in \Gal(K/F)$) and $\tau$ be the automorphisms of $L$ over $F$ defined by
			\[ 	\sigma_{\alpha} ((k_1, \dots, k_n)) = (\alpha(k_1), \dots, \alpha(k_n)) \text{ and }
				\tau((k_1, \dots, k_n)) = (k_n, k_1, \dots, k_{n-1}), 
			\]
			and $G$ be the group generated by $\sigma_{\alpha}$ ($\alpha \in \Gal(K/F)$) and $\tau$. Then $G$ is isomorphic to $\Gal(K/F) \times \mathbb{Z}/n\mathbb{Z}$ and $L$ is a Galois $G$-algebra over $F$. 
	\end{enumerate}
\end{remark}

The following lemma follows from Proposition~\ref{Proposition Galois G algebra} and Remark~\ref{Remark direct product of cyclic extensions}.
\begin{lemma} \label{Lemma SG point -> G-Galois point}
	Let $C=\cup_{i=1}^n C_i \subset \mathbb{P}^2$ be a reduced plane curve and $P \in \mathbb{P}^2$ be a point. We assume that $P$ is an SG point. Let $H$ be a group isomorphic to the Galois group, given by each $\pi_P : C_i \dashrightarrow \mathbb{P}^1$. Then $P$ can be considered a Galois point with Galois group $G \simeq H \times\mathbb{Z}/n \mathbb{Z}$. If $H$ is cyclic, then $P$ can also be considered a Galois point with Galois group $G \simeq \mathbb{Z}/dn \mathbb{Z}$.
\end{lemma}

By Lemma~\ref{Lemma SG point -> G-Galois point}, all SG points in the examples in Sections 3, 4, and 5 can be regarded as Galois points. The following four lemmas are provided by the definition of SG point.

\begin{lemma} \label{Lemma Geometric property of SG point}
	Let $C=\cup_{i=1}^n C_n \subset \mathbb{P}^2$ be a reduced plane curve. A point $P \in \mathbb{P}^2$ is an SG point iff	\begin{enumerate}
		\item $P$ is a Galois point for each $C_i$ ($i = 1, \dots, n$), and
		\item there exist birational maps $\phi_{i,j} : C_i \dashrightarrow C_j$ for every $1 \leq i < j \leq n$ such that $\pi_P \circ \phi_{i,j} = \pi_P$. 
	\end{enumerate}  
\end{lemma}

\begin{lemma} \label{Lemma SG points for n or 2 components}
		Let $C=\cup_{i=1}^n C_n \subset \mathbb{P}^2$ be a reduced plane curve. A point $P \in \mathbb{P}^2$ is an SG point for $C$ iff $P$ is an SG point for $C_i \cup C_j$ for any $i,j \in \{1, \dots, n\}$ with $i \ne j$.  
\end{lemma}

\begin{lemma} \label{Lemma number of SG points}
	Let $C=\cup_{i=1}^n C_n \subset \mathbb{P}^2$ be a reduced plane curve. Then 
	\[ \Delta(C_1) \supset S\Delta(C_1, C_2) \supset \dots \supset  S\Delta(C_1, \dots, C_n). \]
	Assume that the component $C_1$ is a nonsingular curve of degree $d_1$. If $d_1 \geq 3$, then $\# S\Delta_{\mathrm{out}}(C_1, \dots, C_n) \leq 3$. If $d_1 = 4$, then $\# S\Delta_{\mathrm{in}}(C_1, \dots, C_n) \leq 4$. If $d_1 \geq 5$, then $\# S\Delta_{\mathrm{in}}(C_1, \dots, C_n) \leq 1$.
\end{lemma}

\begin{lemma}\label{Lemma set of inner/outer SG points} 
	Let $C=\cup_{i=1}^n C_n \subset \mathbb{P}^2$ be a reduced plane curve. 
	Assume that every $C_i$ is nonsingular. Then 
	\[S\Delta(C_1, \dots, C_n) = S\Delta_{\mathrm{in}}(C_1, \dots, C_n) \cup S\Delta_{\mathrm{out}}(C_1, \dots, C_n),\] 
	\[S\Delta_{\mathrm{in}}(C_1, \dots, C_n) \subset \cap_{i=1}^n \Delta_{\mathrm{in}}(C_i), \text{ and}\] 
	\[S\Delta_{\mathrm{out}}(C_1, \dots, C_n) \subset \cap_{i=1}^n \Delta_{\mathrm{out}}(C_i).\] 
\end{lemma}

The following lemma is implicitly used in frequent discussions on replacing the coordinates of the projective plane to simplify the definition equations of curves or the representation matrices of the Galois groups.

\begin{lemma} \label{Lemma the image of SG point}
	For a projective transformation $\sigma$ of $\mathbb{P}^2$, the map 
	\[ S\Delta(C_1,\dots,C_n) \rightarrow S\Delta(\sigma(C_1), \dots, \sigma(C_n)) \text{ given by } \hspace{1ex} P \mapsto \sigma(P) \]
	is bijective.  
\end{lemma}

\begin{proof}
	Let $P$ be an SG point for $C=\cup_{i=1}^n C_n$. If $\phi:  C_i \dashrightarrow C_j $ is a birational map with $\pi_P \circ \phi = \pi_P$ for some  $i,j \in \{1, \dots, n\}$, then $\pi_{\sigma(P)} \circ (\sigma \circ \phi \circ \sigma^{-1}) = \pi_{\sigma (P)}$. Hence, $\sigma (P)$ is an SG point for $\sigma(C)$. 
\end{proof}

According to the following lemma, if a reduced curve consists of nonsingular components, then birational maps between the irreducible components appearing in the representations of the Galois groups are given by projective transformations of the projective plane.

\begin{lemma} \label{Lemma birat map extend to Proj trans}
	Let $C_1, C_2 \subset \mathbb{P}^2$ be nonsingular projective curves of same degree $d$. Let $\phi:C_2 \rightarrow C_1$ be an isomorphism. 
	\begin{enumerate}
		\item If $d \geq 4$, or;
		\item if $d=3$ and there is a point $P$ such that $\pi_P \circ \phi = \pi_P$ as a map $C_2 \rightarrow \mathbb{P}^1$, 
	\end{enumerate}
	then $\phi$ is a restriction of some projective transformation of $\mathbb{P}^2$. 
\end{lemma}
\begin{proof}
	For Case (1), according to \cite{zbMATH03629105} and Appendix A 18 in \cite{zbMATH03891516}, there is a unique linear system $g^2_d$ of degree $d$ and projective dimension $2$ defined on $C_2$. The identity morphism $\id_{C_2}: C_2 \rightarrow C_2 \subset \mathbb{P}^2$ and morphism $\phi: C_2 \rightarrow C_1 \subset \mathbb{P}^2$ are both determined by $g^2_d$. In Case (2), since $\pi_P \circ \phi = \pi_P$, we know that $\id_{C_2}$ and $\phi$ are determined by the same complete linear system. Thus, in both cases, there exists a projective transformation $T$ of $\mathbb{P}^2$ such that $T \circ \id_{C_2} = \phi$. 
\end{proof}


\section{Outer SG points for a plane curve consisting of quadrics}

In this section, we investigate the number of outer SG points for a plane curve consisting of two nonsingular quadrics. 

\begin{theorem}\label{Theorem SG points two quadrics}
	Let $C = C_1 \cup C_2 \subset \mathbb{P}^2$ be a reduced plane curve. Assume that $C_1$ and $C_2$ are nonsingular quadratic curves. Then $\# S\Delta_{\mathrm{out}}(C_1, C_2) \in \{0,1,3,6\}$.   
\end{theorem}
\begin{proof}
	Let $\hat{\mathbb{P}}^2$ be the dual projective plane and $\hat{C}_1$ and $\hat{C}_2$ be the dual curves of $C_1$ and $C_2$, respectively. Let  $\mathcal{L}(\hat{C}_1, \hat{C}_2)$ be the set $\{ \overline{pq} \subset \hat{\mathbb{P}}^2 \mid p, q \in \hat{C}_1 \cap \hat{C}_2, p \ne q \}$, where $\overline{pq}$ is the line passing through points $p$ and $q$. Then we have a map 
			$\varphi: S\Delta_{\mathrm{out}} (C_1, C_2) \rightarrow \mathcal{L}(\hat{C}_1, \hat{C}_2)$
			as follows. For $P \in S\Delta_{\mathrm{out}} (C_1, C_2)$, there exists an isomorphism $\phi: C_2 \rightarrow C_1$ such that $\pi_P \circ \phi = \pi_P$. Thus, there exist just two lines $l_p$ and $l_q$ in $\mathbb{P}^2$ passing through $P$ and tangent to the two curves $C_1$ and $C_2$. Let $p$ and $q$ be the points in $\hat{\mathbb{P}}^2$ that correspond to lines $l_p$ and $l_q$, respectively. $\overline{pq} \in \mathcal{L}(\hat{C}_1, \hat{C}_2)$. We set $\varphi (P) = \overline{pq}$. 
			As $\overline{pq}$ in $\hat{\mathbb{P}}^2$ is the line corresponding to the point $P$ in $\mathbb{P}^2$, $\varphi$ is injective. 
			
	We demonstrate that $\varphi$ is surjective. For a line $\overline{pq} \in \mathcal{L}(\hat{C}_1, \hat{C}_2)$, let $P$ be the point in $\mathbb{P}^2$ corresponding to $\overline{pq}$. Because $\pi_P:C_1 \rightarrow \mathbb{P}^1$ and $\pi_P:C_2 \rightarrow \mathbb{P}^1$ are of degree $2$, $P$ is a Galois point for $C_1$ and for $C_2$. Since the branch loci of $\pi_P:C_1 \rightarrow \mathbb{P}^1$ and $\pi_P:C_2 \rightarrow \mathbb{P}^1$ are the same, we have that the zero loci of the discriminants of the quadratic extensions $k(C_1) / \pi_P^*k (\mathbb{P}^1)$ and $k(C_2) / \pi_P^*k (\mathbb{P}^1)$ are the same. Thus, $k(C_1) \simeq k(C_2)$ over $k(\mathbb{P}^1)$. The point $P$ is an SG point for $C$.  
	
	As the dual curves $\hat{C}_1$ and $\hat{C}_2$ are nonsingular quadrics, $\# (\hat{C}_1 \cap \hat{C}_2) \in \{1,2,3,4\}$. We have that $\# S\Delta_{\mathrm{out}}(C_1, C_2)= \# \mathcal{L}(\hat{C}_1, \hat{C}_2) \in \{0,1,3,6\}$.  
\end{proof}

\begin{example} 
	Let $C_1$ and $C_2$ be the nonsingular quadric curves defined by $X^2+Y^2-Z^2=0$ and $X^2+Y^2-4YZ+3Z^2=0$, respectively. Then the dual curves $\hat{C}_1$ and $\hat{C}_2$ are denoted as $X^2+Y^2-Z^2=0$ and $X^2-3Y^2-4YZ-Z^2=0$, respectively. Thus,
	\[ \hat{C}_1 \cap \hat{C}_2 = \{ R_1:=(1:0:1), R_2:=(-1:0:1), R_3:=(0:-1:1) \}. \]
	In the dual projective plane $\hat{\mathbb{P}}^2$, we have that $\overline{R_1 R_2}:Y=0$, $\overline{R_2 R_3}:X+Y+Z=0$ and $\overline{R_3 R_1}:-X+Y+Z=0$. These three lines correspond to the SG points in $\mathbb{P}^2$, that is,  
	\[ S\Delta_{\mathrm{out}}(C_1,C_2) = \{P_1:=(0:1:0), P_2:=(1:1:1), P_3:=(-1:1:1) \}. \] 
	It is clear that three lines $l:X+Z=0$, $l':X-Z=0$ and $l^{\sharp}: Y-Z = 0$ in $\mathbb{P}^2$, which are corresponding to the three points $R_1$, $R_2$ and $R_3$ in $\hat{\mathbb{P}}^2$, respectively, tangent to two curves $C_1$ and $C_2$, and satisfy $l \cap l' = {P_1}$, $l' \cap l^{\sharp} = \{P_2\}$ and $l^{\sharp} \cap l = \{P_3\}$. 
\end{example}

\begin{example}
	Let $C_1: X^2 - 4YZ = 0$ and $C_2:X^2  +4Y^2 - 4YZ =0$. Then the dual curves of these two curves are $\hat{C}_{1} : X^2 - YZ=0$ and 
$\hat{C}_{2}:  X^2 - YZ - Z^2=0$. Because $\hat{C}_{1} \cap\hat{C}_{2} = \{ (0 : 1 : 0) \}$, we see that $S\Delta_{\mathrm{out}}(C_1,C_2)=\emptyset$. 
\end{example}

\begin{example}
	Let $C_1: X^2+Y^2-Z^2=0$ and $C_2:Y^2 + 2(X+Z)(X+Y)=0$. Then $S\Delta_{\mathrm{out}}(C_1,C_2)=\{ (0:1:0) \}.$
\end{example}

\begin{example}
Let $C_i$ ($i=1, \dots n$) be a plane curve defined by 
\[ X^2 + \frac{1}{i} Y^2 - \frac{1}{1+i} Z^2 = 0.\]
Let $C$ be the reduced curve $\cup_{i=1}^n C_i$. 
Then the dual curves $\hat{C}_i$ are given by $X^2+iY^2-(1+i)Z^2 = 0$. For any $i,j \in \{1,\dots,n\}$ with $i \ne j$, $\hat{C}_i \cap \hat{C}_j = \{(1:1:1),(-1:1:1),(1:-1:1),(1:1:-1)\}$. Hence, $S\Delta_{\mathrm{out}}(C_i, C_j) = \{ (0:1:1),(1:0:1),(1:1:0),(0:-1:1),(-1:0:1),(-1:1:0)\}$. By Lemma~\ref{Lemma SG points for n or 2 components}, the set is equal to $S\Delta_{\mathrm{out}}(C_1, \dots,  C_n)$.
\end{example}


\section{Outer SG points for a plane curve consisting of curves of the same degree $d \geq 3$ }

In this section, we prove the theorem below and present an example of a reduced plane curve formed by two nonsingular curves of the same degree $d \geq 3$, which has an outer SG point.

\begin{theorem} \label{Theorem the number of outer SG points d>2}
	Let $C=C_1 \cup C_2 \subset \mathbb{P}^2$ be a reduced curve. Assume that $C_1$ and $C_2$ are nonsingular curves of the same degree $d \geq 3$. Then $\# S\Delta_{\mathrm{out}}(C_1, C_2) \leq 1$. 
\end{theorem}
\begin{proof}
	We assume that $\# S\Delta_{\mathrm{out}}(C_1, C_2) \geq 2$. We denote the three points $(1:0:0)$, $(0:1:0)$ and $(0:0:1)$ by $P_1$, $P_2$ and $P_3$, respectively. From Theorem~\ref{Theorem Galois points} (3) and (4) and Lemmas~\ref{Lemma Geometric property of SG point} and \ref{Lemma birat map extend to Proj trans}, we have that $C_1$ and $C_2$ are projectively equivalent to the Fermat curve $F_d: X^d+Y^d+Z^d=0$. By performing projective transformations, we assume that $C_1 = F_d$. Moreover, we assume that $\{P_1, P_2\} \subset S\Delta_{\mathrm{out}}(C_1,C_2)$. By Lemmas~\ref{Lemma Geometric property of SG point} and \ref{Lemma birat map extend to Proj trans} again, there exist projective transformations $\phi$ and $\psi$ such that $\phi(C_2)=C_1$, $\psi(C_2)=C_1$, $\pi_{P_1} \circ \phi = \pi_{P_1}$ and $\pi_{P_2} \circ \psi = \pi_{P_2}$. We observe that $\phi(P_1) = P_1$ and $\psi(P_2) = P_2$.
	
	\begin{claim} \label{Claim phi(P_2) = P_2 psi(P_1) = P_1}
		$\phi(P_2) = P_2$ and $\psi(P_1) = P_1$. 
	\end{claim}
	We prove Claim~\ref{Claim phi(P_2) = P_2 psi(P_1) = P_1}. From Lemma~\ref{Lemma set of inner/outer SG points}, $P_2$ is an outer Galois point for $C_2$. Thus, $\phi(P_2)$ is an outer Galois point for $C_1=\phi(C_2)$. From Theorem~\ref{Theorem Galois points} (4),  $\phi(P_2) \in \{P_1, P_2, P_3\}$. As $\pi_{P_1} \circ \phi = \pi_{P_1}$, $\phi (P_2)$ is on the line passing through points $P_1$ and $P_2$. Because $P_1$, $P_2$, and $P_3$ are not collinear and $\phi(P_1) = P_1$, we have $\phi(P_2) = P_2$. Similarly, we obtain that $\psi(P_1) = P_1$.  
	
	From Claim~\ref{Claim phi(P_2) = P_2 psi(P_1) = P_1} and some elementary matrix calculations, we obtain 
	\[	\phi = 	\begin{pmatrix}
					\lambda & 0 & \alpha \\
					0 & 1 & 0 \\
					0 & 0 & 1
				\end{pmatrix} \text{ and }
		\psi =	\begin{pmatrix}
					1 & 0 & 0 \\
					0 & \mu & \beta \\
					0 & 0 & 1
				\end{pmatrix}, 
	\]
	where $\lambda, \mu \in k\setminus\{0\}$ and $\alpha, \beta \in k$. 
	\begin{claim} \label{Claim alpha, beta, lambda, mu}
		$\alpha = \beta = 0$ and $\lambda^d=\mu^d=1$.
	\end{claim}
	Let us prove Claim~\ref{Claim alpha, beta, lambda, mu}. Curves $\phi^{-1}(C_1)$ and $\psi^{-1}(C_2)$ are expressed as $F:=(\lambda X + \alpha Z)^d + Y^d + Z^d =0$ and $G:= X^d + (\mu Y + \beta Z)^d + Z^d = 0$, respectively. As $C_2=\phi^{-1}(C_1)=\psi^{-1}(C_1)$, there exists $\xi \in k\setminus \{ 0 \}$ such that $F = \xi G$. By comparing both sides of each term, we observe that $\alpha = \beta = 0$ and $\lambda^d=\mu^d=1$.
	
	According to Claim~\ref{Claim alpha, beta, lambda, mu}, the curve $C_2$ must coincide with the curve $C_1$: However, this finding is contradictory. 
\end{proof}

\begin{example}
	Let $C_i \subset \mathbb{P}^2$ ($i = 1, \dots, n$) be the plane curves defined by $X^d + \zeta_n^i Y^d + Z^d = 0$, where $d \geq 3$ and $\zeta_n$ is a primitive $n$th root of unity. Let $C$ be the reduced plane curve $\cup_{i=1}^n C_i$. Then the point $P:=(0:1:0)$ is an SG point for $C$. Indeed, from Theorem~\ref{Theorem Galois points}~(4), $P$ is a Galois point for $C_n$. Since the projective transformations 
	\[\sigma^i:=\begin{pmatrix}
		1 & 0 & 0 \\ 0 & \zeta_{nd}^i & 0 \\ 0 & 0 & 1
	\end{pmatrix}, \text{ where $\zeta_{nd}$ is a primitive ($nd$)th root of unity}, \]
	provide $\sigma^i: C_i \rightarrow C_n$, $P$ is also a Galois point for $C_i$. Moreover, $\sigma^i$ satisfy $\pi_P \circ \sigma^i = \pi_P$.   
\end{example}

\section{Inner SG points for a plane curve consisting of quartic curves}

In this section, we consider the inner SG points for a reduced plane curve $C$ consisting of nonsingular components. If the components of $C$ contain a nonsingular curve of degree $d \geq 5$, then the number of SG points is at most one, according to Theorem~\ref{Theorem Galois points} (5) and Lemma~\ref{Lemma number of SG points}. We consider the number of SG points for a reduced plane curve $C$ consisting of nonsingular quartic curves. We denote by $C_1^{(1)}$, $C_2^{(1)}$, $C_2^{(2)}$ and $C_2^{(3)}$ the quartic curves defined by 
	\[ \begin{split}	
		&C_1^{(1)}: XY^3+X^4+Z^4=0, \\
		&C_2^{(1)}: X((\zeta_4-1)X+\zeta_4 Y)^3+X^4+Z^4=0, \\
		&C_2^{(2)}: X(-2X-Y)^3+X^4+Z^4=0 \text{ and}\\
		&C_2^{(3)}: X((-\zeta_4-1)X-\zeta_4Y)^3+X^4+Z^4=0, 
		\end{split}
	\]
where $\zeta_4$ denotes a primitive 4th root of unity. 

\begin{theorem}\label{Theorem the number of inner SG points}
	Let $C=C_1 \cup C_2 \subset \mathbb{P}^2$ be a reduced plane curve consisting of two nonsingular quartic curves. 
	\begin{enumerate}
		\item The number of inner SG points for $C$ is $0$, $1$ or $2$.
		\item If $\# S\Delta_{\mathrm{in}}(C_1, C_2) = 2$, then there exists a projective transformation $T$ of $\mathbb{P}^2$ such that $T(C_1)=C_1^{(1)}$ and $T(C_2) \in \{ C_2^{(1)}, C_2^{(2)}, C_2^{(3)}\}$. 
		\item Assume that $C_1=C_1^{(1)}$. Then $C_2 \in \{ C_2^{(1)}, C_2^{(2)}, C_2^{(3)}\}$ iff $S\Delta_{\mathrm{in}}(C_1, C_2) = \{ (0:1:0), (-1:1:0) \}$.
	\end{enumerate}		  
\end{theorem}

\begin{lemma} \label{Lemma repres matrix inner SG points} 
	Let $C_1$ be the curve $C_1^{(1)}$. Let  $P_1$ and $P_2$ be the points $(0:1:0)$ and $(-1:1:0)$, respectively. Let $\sigma_1$ and $\sigma_2$ be projective transformations of $\mathbb{P}^2$. Then the following two conditions are equivalent: 
	\begin{enumerate}
		\item $\pi_{P_i} \circ \sigma_i = \pi_{P_i}$ for $i=1,2$, $\sigma_i(P_j)=P_j$ for all $1 \leq i,j \leq 2$, and $\sigma_1^{-1}(C_1)=\sigma_2^{-1}(C_1)$; 
		\item $\sigma_1=\begin{pmatrix}
			1 & 0 & 0 \\
			a-1 & a & 0 \\
			0 & 0 & 1
		\end{pmatrix}$ and 
		$\sigma_2=\begin{pmatrix}
			c & 0 & 0 \\
			-c+1 & 1 & 0 \\
			0 & 0 & 1
		\end{pmatrix}$, where $a, c\in k$, $a^4=1$ and $c=a^3$.
	\end{enumerate} 
\end{lemma}
\begin{proof}
	The implication $(2) \Rightarrow (1)$ is clear. We now show that $(1) \Rightarrow (2)$. Assume that condition (1) holds. From some elementary matrix calculations, we obtain 
	\[	\sigma_1=\begin{pmatrix}
			1 & 0 & 0 \\
			a-1 & a & b \\
			0 & 0 & 1
		\end{pmatrix} \text{ and } 
		\sigma_2=\begin{pmatrix}
			c & 0 & d \\
			-c+1 & 1 & -d \\
			0 & 0 & 1
		\end{pmatrix},\] 
	where $a,b,c,d \in k$. The curves $\sigma_1^{-1}(C_1)$ and $\sigma_2^{-1}(C_1)$ are given by 
	\[ F_1 :=X ((a-1)X + aY + bZ)^3 +X^4+Z^4=0 \text{ and} \] 
	\[F_2 := (cX +dZ) ((-c+1)X+Y-dZ)^3 + (cX+dZ)^4 + Z^4 =0.\]
	Because $\sigma_1^{-1}(C_1)=\sigma_2^{-1}(C_1)$, there exists $f \in k \setminus \{0\}$ such that $F_1=fF_2$. By comparing both sides of each term, $a^4=1$, $c=a^3$ and $b=d=0$. 
\end{proof}

Note that $\sigma_1$ and $\sigma_2$ in (2) of Lemma~\ref{Lemma repres matrix inner SG points} satisfy $\sigma_1 \sigma_2 =\sigma_2 \sigma_1$. 

\begin{proof}[Proof of Theorem~\ref{Theorem the number of inner SG points}]
	Let us prove assertion (1). We assume $\# S \Delta_{\mathrm{in}}(C_1, C_2) \geq 2$. From Lemma~\ref{Lemma set of inner/outer SG points} and Theorem~\ref{Theorem Galois points} (6) and (7), the curves $C_1$ and $C_2$ are projectively equivalent to the curve $C_1^{(1)}$. By performing projective transformations, we assume that $C_1 = C_1^{(1)}$. Then, by Lemma~\ref{Lemma set of inner/outer SG points} again and Theorem~\ref{Theorem Galois points} (9), $S\Delta_{\mathrm{in}}(C_1, C_2) \subset \{P_1, P_2, P_3, P_4\}$, where $P_1=(0:1:0)$, $P_2=(-1:1:0)$, $P_3=(-\omega:1:0)$ and $P_4=(-\omega^2:1:0)$ ($\omega$ is a primitive cubic root of unity). We assume that $P_1 \in S \Delta_{\mathrm{in}}(C_1, C_2)$ by taking a projective transformation if we need, for example, $\sigma_2$, $\sigma_2^2$, $\sigma_2 \sigma_1$ and $\sigma_2 \sigma_1^2$, where	
	\begin{equation} \label{Equation automorphisms}
			 	\sigma_1:= \begin{pmatrix}
			1 & 0 & 0 \\ 0 & \omega^2 & 0 \\ 0 & 0 & 1 
		\end{pmatrix} \text{ and }
		\sigma_2:=\begin{pmatrix}
			1 & \omega^2 & 0 \\ 2\omega^2 & -\omega & 0 \\ 0 & 0 & \sqrt{3} 
		\end{pmatrix}.  
	\end{equation}
	 (Note that $\sigma_i$ ($i=1,2$) induces an automorphism of $C_1$ associated with the Galois point $P_i$.)   
 	Moreover, we assume that $P_2 \in S \Delta_{\mathrm{in}}(C_1, C_2)$ by considering the projective transformation $\sigma_1$ or $\sigma_1^2$.  
	By Lemmas~\ref{Lemma Geometric property of SG point} and \ref{Lemma birat map extend to Proj trans}, there exist projective transformations $\phi_i$ ($i=1,2$) such that $\phi_i(C_2) = C_1$ and $\pi_{P_i} \circ \phi_i=\pi_{P_i}$. Clearly $\phi_i(P_i)=P_i$ ($i=1,2$). If $\phi_1(P_2) \ne P_2$, by replacing $\phi_1$ with another projective transformation, we assume $\phi_1(P_2) = P_2$. Indeed, if $\phi_1(P_2) \ne P_2$, then since $\phi_1(P_2) \in \Delta(C_1)$ we have $\phi_1(P_2) \in \{ P_3, P_4 \}$. Therefore, we replace $\phi_1$ with $\sigma_1^2 \circ \phi_1$ or $\sigma_1 \circ \phi_1$. For $\phi_2$, by a similar argument above, we also may assume that $\phi_2(P_1)=P_1$.  
	By Lemma~\ref{Lemma repres matrix inner SG points}, we have 
	\begin{equation} \label{Equation two inner SG points}
		\phi_1 = 	\begin{pmatrix}
						1 & 0 & 0 \\ a-1 & a & 0 \\ 0 & 0 & 1
					\end{pmatrix}, \text{ where $a \in k$ and $a^4 =1$.}
	\end{equation} 
	Hence, $\phi_1(P_3)=(-\omega: -\omega (a-1)+a:0)$. We assume $\# S \Delta_{\mathrm{in}}(C_1, C_2) \geq 3$. When $S\Delta_{\mathrm{in}}(C_1, C_2) \supset \{ P_1, P_2, P_3 \}$, we have $P_3 \in \Delta(C_2)$ hence, $\phi_1(P_3) \in \Delta(C_1)$. Because $\phi_1(P_1) = P_1$ and $\phi_1(P_2)=P_2$, we have $\phi_1(P_3) \in \{ P_3, P_4 \}$. If $\phi_1(P_3) = P_3$, then $(-\omega: -\omega (a-1)+a:0) = (-\omega : 1 : 0)$; that is, $a=1$ and $\phi_1$ is the identity transformation. However, this finding is contradictory. If $\phi_1(P_3) = P_4$, then $(-\omega: -\omega (a-1)+a:0) = (-\omega^2 : 1 : 0)$, that is, $a=-\omega$. This contradicts $a^4=1$. In the case where $S\Delta_{\mathrm{in}}(C_1, C_2) \supset \{ P_1, P_2, P_4 \}$, a similar argument leads to a contradiction. We conclude that $\# S \Delta_{\mathrm{in}}(C_1, C_2) = 2$. 

	Let us prove assertion (2). Based on the argument above, there exists a projective transformation $T$ such that $T(C_1)=C_1^{(1)}$ and $S\Delta_{\mathrm{in}}(T(C_1), T(C_2))= \{ P_1, P_2\}$. Moreover, $\phi_1$ in Eq.~(\ref{Equation two inner SG points}) satisfies $T(C_2)=\phi_1^{-1}(T(C_1))$. We conclude that $T(C_2)$ is $C_2^{(1)}$, $C_2^{(2)}$ or $C_2^{(3)}$. 
	
	Let us prove assertion (3). We assume $C_1 = C_1^{(1)}$. Let $\phi_1^{(j)}$ and $\phi_2^{(j)}$ ($j=1,2,3$) be the projective transformations represented by 
	\[ 	\begin{pmatrix}
			1 & 0 & 0 \\ \zeta_4^j-1 & \zeta_4^j & 0 \\ 0 & 0 & 1
		\end{pmatrix} \hspace{1ex} \text{and} \hspace{1ex}
		\begin{pmatrix}
			\zeta_4^{3j} & 0 & 0 \\ -\zeta_4^{3j}+1 & 1 & 0 \\ 0 & 0 & 1
		\end{pmatrix}, \hspace{1ex} \text{respectively}.
	\] 
	Assume that $C_2=C_2^{(j)}$ ($j=1,2,3$). Then, since $P_1, P_2 \in \Delta(C_1)$, $\phi_1^{(j)} (C_2) = C_1$, $\phi_1^{(j)}(P_1) = P_1$ and $\phi_1^{(j)}(P_2) = P_2$, we have $P_1, P_2 \in \Delta(C_2)$. In addition, $\pi_{P_1} \circ \phi_1^{(j)} = \pi_{P_1}$ and $\pi_{P_2} \circ \phi_2^{(j)} = \pi_{P_2}$. Thus, $P_1$ and $P_2$ are the inner SG points for $C_1 \cup C_2$. Because $\# S\Delta_{\mathrm{in}}(C_1, C_2) \leq 2$, we have $S\Delta_{\mathrm{in}}(C_1, C_2)=\{ P_1, P_2 \}$. The opposite can be deduced from the arguments presented in the proofs of assertions (1) and (2). 
\end{proof}

\begin{corollary}
	Let $C=\cup_{i=1}^n C_i \subset \mathbb{P}^2$ be a reduced curve. Assume that $n \geq 2$ and all $C_i$ are nonsingular with the same degree $d \geq 4$. If $\# S\Delta_{\mathrm{in}}(C_1, \dots, C_n) \geq 2$, then $n \leq 4$, $d=4$, $\# S\Delta_{\mathrm{in}}(C_1, \dots, C_n) = 2$, and $C$ is projectively equivalent to a curve whose components are curves $C_1^{(1)}$, $C_2^{(1)}$, $C_2^{(2)}$ or $C_2^{(3)}$. 
\end{corollary}

\begin{example}
	Let $C_i$ ($i=1, \dots, n$) be the plane curves defined by $\zeta_n^i XY^{d-1} + X^d + Z^d =0$, where $d \geq 4$ and $\zeta_n$ is a primitive $n$th root of unity. Let $C$ be the reduced curve $\cup_{i=1}^n C_i$. Then the point $P:=(0:1:0)$ is an SG point. Indeed, from Theorem~\ref{Theorem Galois points}~(9) and (10), $P$ is a Galois point for $C_n$. Since projective transformations 
	\[\sigma^i:=\begin{pmatrix}
		1 & 0 & 0 \\ 0 & \zeta_{n(d-1)}^i & 0 \\ 0 & 0 & 1
	\end{pmatrix}, \text{ where $\zeta_{n(d-1)}$ is a primitive ($n(d-1)$)st root of unity}, \]
	give $\sigma^i: C_i \rightarrow C_n$, $P$ is also a Galois point for $C_i$. Moreover, $\sigma^i$ satisfy $\pi_P \circ \sigma^i = \pi_P$. 
\end{example}

By an argument similar to the proof of Theorem~\ref{Theorem the number of inner SG points}, we can observe the following: 
	
\begin{theorem} \label{Theorem C has inner SG and outer SG}
	For a reduced curve $C= \cup_{i=1}^n C_i \subset \mathbb{P}^2$, if $C_1$ and $C_2$ are nonsingular curves with degrees $d_1 \geq 4$ and  $d_2 \geq 4$, then $\# S\Delta_{\mathrm{in}}(C_1, \dots, C_n) = 0$ or $\# S\Delta_{\mathrm{out}}(C_1, \dots, C_n) = 0$. 
\end{theorem}
\begin{proof}
	Assume that $\# S\Delta_{\mathrm{in}}(C_1, \dots, C_n) \geq 1$ and $\# S\Delta_{\mathrm{out}}(C_1, \dots, C_n) \geq 1$. From Lemma~\ref{Lemma number of SG points}, $\# S\Delta_{\mathrm{in}}(C_1, C_2) \geq 1$ and $\# S\Delta_{\mathrm{out}}(C_1, C_2) \geq 1$. By Lemma~\ref{Lemma number of SG points} again, Definition~\ref{Def of SG point} and Theorem~\ref{Theorem Galois points}~(8), we observe that $d_1=d_2 (=:d)$, and $C_1$ and $C_2$ are projectively equivalent to the curve $XY^{d-1} + X^d + Z^d=0$. By performing  projective transformations, we assume that $C_1$ is given by $XY^{d-1} + X^d + Z^d=0$. Let $P$ and $Q$ be inner and outer SG points, respectively, for $C_1 \cup C_2$. By Theorem~\ref{Theorem Galois points}~(9) and (10), if $d \geq 5$, then $P=(0:1:0)$ and $Q=(0:0:1)$; and if $d=4$, then $P \in \{ (0:1:0), (-1:1:0), (-\omega:1:0), (-\omega^2:1:0)\}$ and $Q=(0:0:1)$, where $\omega$ is a primitive cubic root of unity. If $d=4$, then by perfoming projective transformations represented as Eq.~(\ref{Equation automorphisms}), which are automorphisms of $C_1$, we assume $P=(0:1:0)$. By Lemmas~\ref{Lemma Geometric property of SG point} and \ref{Lemma birat map extend to Proj trans}, there exist projective transformations $\phi_P$ and $\phi_Q$ such that $\phi_P(C_2)=C_1$, $\phi_Q(C_2)=C_1$, $\pi_{P} \circ \phi_P = \pi_{P}$ and $\pi_{Q} \circ \phi_Q = \pi_{Q}$. Because $Q$ is a unique outer Galois point for $C_1$ and $C_2$, we obtain $\phi_P(Q)=Q$. Since $\phi_Q(P)$ is an inner Galois point for $C_1$, $\phi_Q(P)$ is on line $Z=0$. From some elementary matrix calculations, we can observe that 
	\[	\phi_P = \begin{pmatrix}
			1 & 0 & 0 \\
			\alpha & \beta & 0 \\
			0 & 0 & 1
		\end{pmatrix} \text{ and }
		\phi_Q = \begin{pmatrix}
			1 & 0 & 0 \\
			0 & 1 & 0 \\
			\gamma & 0 & \delta
		\end{pmatrix}, 
	\]
	where $\alpha, \gamma, \in k$ and $\beta, \delta \in k \setminus \{0 \}$. Curves $\phi_P^{-1}(C_1)$ and $\phi_Q^{-1}(C_1)$ are given by 
	$F := X(\alpha X + \beta Y)^{d-1} + X^d + Z^d =0$ and $G:=XY^{d-1}+X^d+(\gamma X+ \delta Z)^d = 0$. Because $C_2=\phi_P^{-1}(C_1)=\phi_Q^{-1}(C_1)$, there exists $\lambda \in k \setminus \{ 0 \}$ such that $F = \lambda G$. By comparing both sides of each term, $\alpha = 0$, $\beta^{d-1}=1$, $\gamma=0$ and $\delta^d=1$. Thus, $C_2 = C_1$. However, this finding is contradictory.
\end{proof}


At the end of this paper, we raised some problems.

\begin{problem}
	Let $C = \cup_{i=1}^n C_i \subset \mathbb{P}^2$ be a reduced plane curve over an algebraically closed field $k$. 
	\begin{enumerate} 
		\item Assume that $\ch (k) = 0$, $n=2$ and both $C_1$ and $C_2$ are nonsingular cubic curves. Then study the number and distribution of inner SG points for $C$. 
		\item Assume that $\ch (k) = 0$, $n=2$ and $C_1$ or $C_2$ has singular points. Then determine the number and distribution of SG points for $C$. 
		\item When $\ch (k) = p > 0$, investigate the number and distribution of SG points for $C$. In particular, study for a curve whose components are projectively equivalent to a Hermitian curve. Note that one of the most famous examples of curves with many Galois points is a Hermitian curve (\cite{zbMATH05126212}).  
	\end{enumerate} 
\end{problem}



\begin{thebibliography}{10}

\bibitem{zbMATH03891516}
E.~Arbarello, M.~Cornalba, P.~A. Griffiths and J.~Harris, {Geometry of algebraic curves: Volume I}, Grundlehren Math. Wiss., vol. 267. Springer-Verlag, New York 1985.

\bibitem{zbMATH03629105}
H.~C. Chang, {On plane algebraic curves}, Chin. J. Math. \textbf{6} (1978), 185--189.

\bibitem{zbMATH01260607}
D.~Eisenbud and J.~Harris, {The geometry of schemes}, Grad. Texts Math., vol. 197, New York, NY: Springer, 2000.

\bibitem{zbMATH05225736}
S.~Fukasawa, {Galois points on quartic curves in characteristic 3}, Nihonkai Math. J. \textbf{17} (2006), 103--110.

\bibitem{zbMATH05202381}
\bysame, {On the number of Galois points for a plane curve in positive characteristic, II}, Geom. Dedicata \textbf{127} (2007), 131--137.

\bibitem{zbMATH05241281}
\bysame, {On the number of Galois points for a plane curve in positive characteristic}, Commun. Algebra \textbf{36} (2008), 29--36.

\bibitem{https://doi.org/10.48550/arxiv.1011.3648}
\bysame, {Complete determination of the number of Galois points for a smooth plane curve}, arXiv:math/1011.3648 (2010).

\bibitem{zbMATH05721541}
\bysame, {On the number of Galois points for a plane curve in positive characteristic, III}, Geom. Dedicata \textbf{146} (2010), 9--20.

\bibitem{zbMATH06199262}
\bysame, {Complete determination of the number of Galois points for a smooth plane curve}, Rend. Semin. Mat. Univ. Padova \textbf{129} (2013), 93--113.

\bibitem{zbMATH06244917}
\bysame, {Galois points for a plane curve in characteristic two}, J. Pure Appl. Algebra \textbf{218} (2014), 343--353.

\bibitem{https://doi.org/10.48550/arxiv.2210.02076}
\bysame, {A new characterisation of the Fermat curve}, arXiv:math/2210.02076 (2022).

\bibitem{zbMATH07199975}
S.~Fukasawa, K.~Miura and T.~Takahashi, {Quasi-Galois points, I: Automorphism groups of plane curves}, Tohoku Math. J. (2) \textbf{71} (2019), 487--494.

\bibitem{https://doi.org/10.48550/arxiv.2211.16033}
\bysame, {Quasi-Galois points, II: Arrangements}, arXiv:math/2211.16033 (2022).

\bibitem{MR217083}
A.~Grothendieck, {\'{E}l\'{e}ments de g\'{e}om\'{e}trie alg\'{e}brique. {I}. {L}e langage des sch\'{e}mas}, Inst. Hautes \'{E}tudes Sci. Publ. Math. (1960), 1--228. 

\bibitem{zbMATH05126212}
M.~Homma, {Galois points for a Hermitian curve}, Commun. Algebra \textbf{34} (2006), 4503--4511.

\bibitem{zbMATH01201676}
M.-A.~Knus, A.~Merkurjev, M.~Rost and J.-P.~Tignol, {The book of involutions. With a preface by J.~Tits}, Colloq. Publ., Am. Math. Soc., vol.~44, Providence, RI: American Mathematical Society, 1998.

\bibitem{zbMATH01441808}
K.~Miura and H.~Yoshihara, {Field theory for function fields of plane quartic curves}, J. Algebra \textbf{226} (2000), 283--294.

\bibitem{zbMATH01643659}
H.~Yoshihara, {Function field theory of plane curves by dual curves}, J. Algebra \textbf{239} (2001), 340--355.

\bibitem{zbMATH01675821}
\bysame, {Galois points on quartic surfaces}, J. Math. Soc. Japan \textbf{53} (2001), 731--743.

\bibitem{zbMATH01956748}
\bysame, {Galois points for smooth hypersurfaces}, J. Algebra \textbf{264} (2003), 520--534.

\bibitem{zbMATH05045279}
\bysame, {Galois lines for space curves}, Algebra Colloq. \textbf{13} (2006), 455--469.

\bibitem{zbMATH05214476}
\bysame, {Galois embedding of algebraic variety and its applications to abelian surface}, Rend. Semin. Mat. Univ. Padova \textbf{117} (2007), 69--85.

\bibitem{ListOfProblems}
H.~Yoshihara and S.~Fukasawa, {List of problems (version: 15 March, 2023)}, Fukasawa Lab., retrieved 27 April, 2023 from \url{https://sites.google.com/sci.kj.yamagata-u.ac.jp/fukasawa-lab/open-questions-english/}.

\end{thebibliography}
\end{document}